\documentclass[journal]{IEEEtran}
\onecolumn
\usepackage{mathrsfs}
\usepackage{amsmath}
\usepackage{graphicx}
\usepackage{amssymb}
\usepackage{enumerate}
\usepackage{longtable,tabularx,float}
\usepackage{cite}
\usepackage{mathcomp}
\usepackage{multirow}
\usepackage{supertabular}
\usepackage{stmaryrd}
\usepackage{color}
\usepackage{url}
\usepackage{mathtools}
\usepackage{tikz}
\usepackage{amsthm}
\newtheorem{definition}{Definition}[section]
\newtheorem{lemma}{Lemma}[section]
\newtheorem{theorem}{Theorem}[section]

\renewcommand{\H}{\mathcal {H}}
\newcommand{\R}{\mathcal {R}}

\newcommand{\mm}{{~\text{mod}^+~}}

\begin{document}

\title{Covering Radius of Permutation Groups with Infinity-Norm}

\author{Xin Wei,~and~Xiande~Zhang
\thanks{X. Wei ({\tt weixinma@mail.ustc.edu.cn}) and X. Zhang ({\tt drzhangx@ustc.edu.cn}) are with School of Mathematical Sciences,
University of Science and Technology of China, Hefei, 230026, Anhui, China.  }
}
\maketitle

\begin{abstract}
\boldmath The covering radius of permutation group codes are studied in this paper with $l_{\infty}$-metric. We determine the covering radius of the $(p,q)$-type group, which is a direct product of two cyclic transitive groups. We also deduce the maximum covering radius among all the relabelings of this group under conjugation, that is, permutation groups with the same algebraic structure but with relabelled members.  Finally, we give a lower bound of the covering radius of the dihedral group code, which differs from the trivial upper bound by a constant at most one. This improves the result of Karni and Schwartz in 2018, where the gap between their lower and upper bounds tends to infinity as the code length grows.
\end{abstract}


\section{introduction}
Given a finite set of points $P$ in a metric space $M$, the {\it covering radius} of $P$ in $M$ is the
smallest number $r$ such that spheres of radius $r$ centered around all the points in $P$ cover the entire space. Such a set $P$ is called a {\it covering code}. There are two central problems in the literature about covering codes: the mathematical
question of determining the covering radius of any given code, and the more practical problem of constructing good covering codes having a specified length and covering radius. An entire book has been written on the subject and we point the reader to \cite[Chap. 1]{cohen1997covering} for a discussion of various applications of covering codes.

In this paper, we focus on the covering radius problem over the symmetric group $S_n$ as a metric space. The study of coding problems over permutations can date back to the works in \cite{SLEPIAN1965PERMUTATION,chadwick1969rank, BERGER1972PERMUTATION}. Since the paper of Blake et al. \cite{blake1979coding}, the symmetric group has been widely studied as a setting for coding theory with various permutation metrics.
 Covering radius for sets of permutations have only been studied by Cameron and Wanless in \cite{cameron2005covering}, with Hamming distance, due to their close relations to classical conjectures of Ryser and Brualdi on transversals of Latin squares \cite{keedwell2015latin}. Recently, more works about the covering radius problem for some permutation groups can be found in \cite{keevash2006random,quistorff2006survey,wanless2013transversals,hendrey2018covering,cossidente2019covering,bamberg2019covering}, but all of which only deal with Hamming distance.


Motivated by applications to information storage in non-volatile memories, the rank-modulation scheme was
suggested \cite{Jiangetal:2009}, in which information is stored in the form of
permutations. The $l_\infty$-metric is one of the main relevant permutation metrics for this scheme to solve a limited-magnitude error model. Thus, a lot of works have done on the error-correcting codes with $l_\infty$-metric recently, see for example \cite{klove2010permutation,tamo2010correcting,tamo2012labeling,zhang2016snake,klove2009generating,klove2011lower,schwartz2011optimal}. For various distances of permutation codes, see a summarization in the survey \cite{DEZA1998METRICS}. For computational complexity problems of finding a particular permutation in a subgroup with some distance property, see \cite{BUCHHEIM2009SUBGROUP} and references therein.

Covering codes over permutations with the $l_\infty$-metric have
 been recently studied in \cite{hassanzadeh2016bounds,wang2015compression}, while the covering radius problem was studied in \cite{karni2018infinity}. A permutation code, as a subset of the symmetric group $S_n$, may happen to be a subgroup, in which case we call it a \emph{group code}. In \cite{karni2018infinity}, the authors studied the covering radius of two permutation group codes, the transitive cyclic group $G_n$  and the dihedral group $D_n$, and they used them as building-block covering codes to get long covering codes, which generalised a construction in \cite{hassanzadeh2016bounds}. Since the $l_\infty$-metric is right invariant, but not left invariant, the authors in \cite{karni2018infinity} also considered a conjugate of the transitive cyclic group, which they call a ``relabeling'', and studied the maximum and minimum covering radius achievable among all relabelings.


Although group structures provide rich information about the codes, it's hard to determine the covering radius in most cases. This is because that the volume of balls under $l_\infty$-metric is not easy to compute, see \cite{klove2009generating,klove2011lower,KLOVE2008SPHERES,schwartz2017improved}. In  \cite{karni2018infinity}, the authors exactly determined the covering radius of $G_n$ as well as the maximum covering radius among all its relabelings, but only gave lower bounds on those for minimum covering radius and for the dihedral groups $D_n$. Their results for relabelings showed that the covering radius of the transitive cyclic group don't increase much after relabelings.

In this paper, we study a new group code, which we call a $(p,q)$-type code, and determine its covering radius with an explicit form. Further, we consider the relabelings of the $(p,q)-$type code, and determine the maximum covering radius that can be achieved. Our result confirms again the fact that relabeling a group code with the $l_\infty$-metric don't change much the covering radius, when preserving the group structure. Finally, we improve the lower bound on the covering radius of $D_n$ given in the end of \cite{karni2018infinity}, by establishing a more precise estimation with a simple expression. The gap between this new lower bound and the trivial upper bound is one for almost all $n$.

The paper is organized as follows. In Section \ref{s:pre}, we introduce formal definitions and notations used through out this paper. Section \ref{s:gpq} is devoted to determination of the covering radius of the $(p,q)-$type group, while Section \ref{s:rel} deals with the relabelings of $(p,q)-$type group. In Section \ref{s:dn}, we give a better lower bound of covering radius of $D_ n$. Finally, we  conclude our results in Section \ref{s:con}.

\section{preliminary}\label{s:pre}
First, we give some useful notations and definitions. Some of them were introduced  in \cite{tamo2012labeling} and \cite{karni2018infinity}.

For integers $m\leq m'$, we denote $[m,m']\triangleq\{m,m+1\cdots,m'\}$, and $[m]\triangleq[1,m]$ for short. When applying this notation to the case $m>m'$, we have an empty set. Let $m\mm n$ denote the unique $r\in[n]$ such that $n$ divides $m-r$.

The symmetric group of permutations over $[n]$ is denoted by $S_n$. For a permutation $f\in S_n$, we use either a one-line notation
for permutations, where $f = [ f_1 , f_2 , \ldots , f_n ]$ denotes a
permutation mapping $i \rightarrow f_i$ for all $i \in [ n ]$, or a cycle notation
$f = ( f_1 , f_2 , \ldots , f_k )$ where $f$ maps $f_i\rightarrow f_{( i + 1 )\mm k}$ for all $i\in [k]$.

The metric we consider in this work is $l_\infty$-metric, which is also called the Chebyshev metric. The distance function is defined by
\begin{displaymath}
d(f,g)\triangleq \max_{i\in[n]}|f(i)-g(i)|
\end{displaymath}for all $f,g\in S_n$. Note that $d$ is right invariant, but not left invariant, see e.g. \cite{DEZA1998METRICS}. For a subset $C\subseteq S_n$, and $f\in S_n$, we define the distance between $f$ and $C$ by
\[d(f,C)\triangleq \min_{g\in C}d(f,g).\]
\begin{definition}
An $(n,M,r)$ {\em covering code} is a subset $C\subseteq S_n$ such that $|C|=M$ and $d(f,C)\le r$ for all $f\in  S_n$. The {\em covering radius} of $C$ is the minimum integer $r$ such that $C$ is an $(n,M,r)$ covering code. We denote it by $r(C)$.
\end{definition}

One of the central problems in this area is to determine or estimate the covering radius $r(C)$ for a given code $C\subseteq S_n$. The most interesting case is when $C$ is a subgroup of $S_n$, for which we refer to $C$ as a {\em group code}. Since the distance function crucially depends on the permuted elements, the ``natural'' and ``relabeling'' descriptions of groups were considered in \cite{tamo2012labeling,karni2018infinity}.

\begin{definition}
For all $n\in \mathbb{N}$, the natural transitive cyclic group, denoted $G_n\subseteq S_n$, is the group generated by the permutation $(1,2,\ldots, n)$, i.e., $G_n\triangleq\left<(1,2,\ldots,n)\right>$. The natural dihedral group $D_n\subseteq S_n$, is defined by
\[D_n\triangleq \left< (1,2,\ldots, n-1,n),\prod_{i=1}^{\lfloor \frac{n}{2}\rfloor}(i,n-i)\right>.\]
\end{definition}

 In \cite{karni2018infinity}, the authors determined the covering radius of transitive cyclic group \[r(G_n)=n-\left\lfloor\frac{\sqrt{4n+1}+1}{2}\right\rfloor,\] and gave bounds for the covering radius of  natural dihedral group  $r(D_n)$. They also considered the non-natural transitive cyclic groups, and studied the covering radius of relabelings of $G_n$. In particular, they consider groups of the form \[G_n^h\triangleq h G_nh^{-1}\triangleq \left<h(1,2,\ldots,n)h^{-1}\right>=\left<(h(1),h(2),\ldots,h(n))\right>\subseteq S_n\] for some $h\in S_n$. Here, $G_n^h$ is called a {\em relabeling} of $G_n$ by conjugation of $h$. In general, we can define the relabeling of any code $C\subseteq S_n$ by \[C^{h}\triangleq hCh^{-1} = \{hgh^{-1}:g\in C\}.\]

\begin{definition}
Let $C\subseteq S_n$ be a covering code. Let $L_{max}(C)$ (respectively, $L_{min}(C)$) denote the maximum (respectively, minimum) achievable covering radius among all relabelings of $C$, i.e., $L_{max}(C)\triangleq \max_{h\in S_{n}}r(C^{h})$ and $L_{min}(C)\triangleq \min_{h\in S_{n}}r(C^{h})$.
\end{definition}

It was shown \cite{karni2018infinity} that \[L_{max}(G_n)=n-\left\lceil\frac{\sqrt{4n+1}-1}{2}\right\rceil\] and \[L_{min}(G_n)\geq n-\left\lceil\sqrt{2n\ln n+2n}\right\rceil.\] This indicates that  the
covering radius of the transitive cyclic group is quite robust under conjugation. While relabeling cannot reduce
the covering radius by much, the downside is that the
covering radius cannot be increased by more than one after
relabeling. A similar question has been asked in \cite{tamo2012labeling} for error-correcting codes $C\subseteq S_n$ and its relabelings $C^h$, but the result is quite different.  It was shown \cite{tamo2012labeling} that the minimum distance of a code
could drastically change due to relabeling, moving from the
minimum possible $1$, to the maximum possible $n-1$, for
some codes. Additionally, every error-correcting code could
be relabeled so that its minimum distance is reduced to
either $1$ or $2$.


Before closing this section, we define several notations commonly used through out this paper.  For any permutation $f\in S_n$ and any integer $j\in[n]$, the {\em position} or the {\em location} of $j$ in $f$ is the unique integer $i\in[n]$ such that $f(i)=j$.  For any $f,g\in S_n$, if $|f(i)-g(i)|\leq r$, then we say that $f$ is {\em $r$-covered by $g$ in position $i$}. Further if $f$ is $r$-covered by $g$ in each position $i\in [n]$, i.e., $d(f,g)\le r$, then we say that $f$ is \emph{$r$-covered} by $g$. On the other hand, if $d(f,g)> r$, then we say that $f$ is \emph{$r$-exposed} by $g$. More precisely, we say that the mapping {\em$i\to f(i)$ is $r$-exposed by $g$}, or $f$ is {\em $r$-exposed by $g$ in position $i$}, if $|f(i)-g(i)|>r$.
Similarly, for a code $C\subset S_n$, if $d(f,C)\le r$, then we say that $f$ is \emph{$r$-covered} by $C$, or $f$ is \emph{$(r,C)$-covered}; if not, then we say that $f$ is {\em $(r,C)$-exposed}.
The same terminologies can be defined to general vectors and codes.

%
%

\section{The Covering Radius of the (p,q)-Type Group}\label{s:gpq}
In this section, we determine the covering radius of a class of groups, which are direct products of two cyclic groups. Here we call it a  $(p,q)$-type group. Generally this type of group isn't transitive or cyclic, but it is cyclic when $\gcd(p,q)=1$. We apply similar idea as in \cite{karni2018infinity} to give an upper bound on the covering radius, and then determine the exact value by construction. 

\begin{definition}
For any $p,q\in\mathbb{N^+}$, the {\em natural $(p,q)$-type} group $G_{p,q}\subseteq S_{p+q}$, is defined by:
\begin{displaymath}
G_{p,q}\triangleq \left<(1,2,\cdots,p),(p+1,p+2,\cdots,p+q)\right>.
\end{displaymath}
\end{definition}

It is easy to see that $|G_{p,q}|=pq$. Without loss of generality, we assume $p\geq q$. For any permutation $g\in G_{p,q}$, if we know the values $g(i)$ and $g(j)$ for some $i\in [p]$ and $j\in [p+1, p+q]$, then we can get all values of $g(l)$ by the following way:
 \begin{equation*}
g(l)=\begin{dcases}
(g(i)+l-i) \mm p, &l\in[p];\\
p+ (g(j)-p+l-j) \mm q, &l\in[p+1,p+q].
\end{dcases}
\end{equation*}

We begin with a rough estimation of the covering radius of $G_{p,q}$. We claim that \[r(G_{p,q})\geq p.\] In fact, consider a permutation $f$ with $f(1)= p+q$ and $f(p+q)=1$ in $S_{p+q}$.  For any permutation $g$ in the group $G_{p,q}$, the largest possible value for $g(1)$ is $p$, while the smallest value that $g(p+q)$ can take is $p+1$. Consequently,  $\min_{g\in G_{p,q}}|f(1)-g(1)|=q$ and $\min_{g\in G_{p,q }}|f(p+q)-g(p+q)|=p$ deduce that $p$ is a lower bound of the covering radius. It's worth mentioning here that we can not obtain a lower bound from the known covering radius $r(G_n)$ and the group isomorphism $G_{p,q}\cong G_p\times G_q$, since the distance function concerns about the maximum absolute values of differences among all coordinates.


 To give an estimate of the upper bound of $r(G_n)$, the authors in \cite{karni2018infinity} defined a set $A_{i\to f(i)}^{H}\triangleq \{h^{-1}(1): i\to f(i) \text{ is $r$-exposed by } h\in H\}$. This set plays a key role on recording all the permutations in $H$ which $r$-expose $f$ in position $i$. If $H$ is a transitive group of size $n$ and each permutation in $H$ is recorded by a set $A_{i\to f(i)}^{H}$ for some $i\in [n]$, then $f$ is $r$-exposed by $H$. In other words, if $\bigcup_{i\in [n]}A_{i\to f(i)}^{H}$ is a proper subset of $[n]$, then $f$ is $r$-covered by $H$. Here, we use similar idea. Since $G_{p,q}$ acts transitively on each subset $[p]$ and $[p+1,p+q]$, we define $A_{i\to f(i)}^C$ separately.


 \begin{definition}\label{extended definition of A}
 Let $C=G_{p,q}$. For a permutation  $f\in S_{p+q}$, and $r\in [p,p+q]$, we define
 \begin{equation*}
A_{i\to f(i)}^C=\begin{dcases}
\{g^{-1}(1): i\to f(i) \text{ is $r$-exposed by  $g$ in $C$} \}, &i\in[p];\\
\{g^{-1}(p+1): i\to f(i) \textrm{ is $r$-exposed by  $g$ in $C$}\}, &i\in[p+1,p+q].
\end{dcases}
\end{equation*}

 \end{definition}

 We are only interested in those positions $i$ such that $A^C_{i\to f(i)}$ are not empty. Define  two sets \[B\triangleq [p+q-r-1]~~~~\text{    and    } ~~~~T\triangleq[r+2,p+q],\] as the bottom and top parts of $[p+q]$ for some implicit $r\geq p$.  Note that $B\subset[p]$, $T\subset[p+1,p+q]$, and $B\cap T=\emptyset$. It is easy to check that if $i\to f(i)$ is $r$-exposed by some $g\in C$, then \[f(i)\in T \text{ if }i\in [p],\text{ and }f(i)\in B\text{ if }i\in [p+1,p+q].\]

 \begin{lemma}\label{counting}
Let $C$ be $G_{p,q}$ and $r\in [p,p+q]$. Then for all $i,j\in[p+q]$,
\begin{equation*}
|A_{i\to j}^C|=\begin{dcases}
             p+q-r-j, &i>p\text{ and }j\in B,  \\
             j-r-1, &i\le p\text{ and }j\in T,\\
             0, &\text{otherwise}.
             \end{dcases}
\end{equation*}
\end{lemma}
\begin{proof}
If $i>p$, then for each $g\in C$, we have $g(i)\in [p+1,p+q]$. If $j\in B$, then $i\to j$ is $r$-exposed by all permutations $g\in C$ with $g(i)\in [j+r+1, p+q]$. Since $G_{p,q}$ has the direct product structure, and acts transitively on $[p+1,p+q]$, there are exactly $p(p+q-r-j)$ permutations $g\in C$ which $r$-expose the map $i\to j$. These permutations can be partitioned into $p+q-r-j$ parts, each consisting of $p$ permutations, which have the same images on the set $[p+1,p+q]$. Each part has a distinct location for the element $p+1$. So $|A_{i\to j}^C|=p+q-r-j$ in this case.

If $i\leq p$ and $j\in T$, then  $i\to j$ is $r$-exposed by  $g\in C$ if and only if $g(i)\in [1, j-r-1]$. By similar arguments, we have $|A_{i\to j}^C|=j-r-1$.

The last case is trivial.
\end{proof}


 \begin{lemma}\label{iff}
Let $f\in S_{p+q}$ be any permutation and $C=G_{p,q}$. Then  $f$ is $r$-exposed by $C$ if and only if at least one of the conditions $[p]= \bigcup_{i\in [p]} A^C_{i \to f(i)}$ and $[p+1,p+q]=\bigcup_{i\in [p+1,p+q]} A^C_{i \to f(i)}$ holds.
\end{lemma}
\begin{proof} For the sufficiency, we only prove it when $[p]= \bigcup_{i\in [p]} A^C_{i \to f(i)}$, the other case is similar. If $f$ is $r$-exposed by some  $g\in C$ in position $i\in [p]$, then $f$ is $r$-exposed by all the $q$ permutations $g'\in C$ satisfying $g'(j)=g(j)$ for all $j\in [p]$. Since $|C|=pq$, it follows that every $g\in C$ $r$-expose $f$, hence $f$ is $(r, C)$-exposed.

For the other direction, we prove it by contradiction. If $[p]\neq \bigcup_{i\in [p]} A^C_{i \to f(i)}$, then due to the transitivity action of $G_{p,q}$ on  $[p]$, there is a permutation $g_1\in C$ such that $f$ is not $r$-exposed by $g_1$ in any position $i\in [p]$, or equivalently, $f$ is $r$-covered by $g_1$ in each position $i\in [p]$.  Similarly,  $[p+1,p+q]\neq\bigcup_{i\in [p+1,p+q]} A^C_{i \to f(i)}$ implies that there exists a permutation $g_2\in C$ such that $f$ is $r$-covered by $g_2$ in each position $i\in [p+1,p+q]$. Now we define a function $g$ as follows:  $g(i)=g_1(i)$ when $i\in[p]$ and $g(i)=g_2(i)$ otherwise. Since  $g_1, g_2\in C$, we have that $g$ is also a permutation in $C$. Since $f$ is $r$-covered by $g\in C$ in all positions, $f$ is $r$-covered by $C$, which is a contradiction.
\end{proof}


Now we are ready to give an upper bound of covering radius of $G_{p,q}$.
\begin{lemma}\label{upper}
For all $p,q\in \mathbb{N^+}$ and $p\geq q$,
\begin{equation*}
r\left(G_{p,q}\right)\leq p+\left\lfloor \left(\sqrt{q+\frac{1}{8}}-\frac{\sqrt{2}}{2}\right)^2-\frac{1}{8} \right\rfloor.
\end{equation*}
\end{lemma}
\begin{proof}
Let $f\in S_{p+q}$ be any permutation. By Lemma~\ref{counting}, we get

\begin{align*} \left |{\bigcup _{i\in [p]} A^C_{i\mapsto f(i)}}\right |\leq&\sum _{i\in [p]} \left |{A^C_{i\mapsto f(i)}}\right | \notag \\=&\sum_{i\in [p]} \max\left\{f(i)-r-1,0\right\} \\\leq &\sum_{j=r+2}^{p+q} \left(j-r-1\right) \\=&\frac{(p+q-r)(p+q-r-1)}{2}. \end{align*}

Similarly, we get

\begin{align*} \left |{\bigcup _{i\in [p+1,p+q]} A^C_{i\mapsto f(i)}}\right |\leq&\sum _{ {i\in [p+1,p+q]}} \left |{A^C_{i\mapsto f(i)}}\right | \notag \\=&\sum_{ {i\in [p+1,p+q]}} \max\left\{p+q-r-f(i),0\right\} \\\leq &\sum_{j=1}^{p+q-r-1} (p+q-r-j) \\=&\frac{(p+q-r)(p+q-r-1)}{2}. \end{align*}

By Lemma~\ref{iff}, if
\begin{equation*}
\frac{(p+q-r)(p+q-r-1)}{2}<\min\left\{p,q\right\}=q,
\end{equation*}
then $f$ is $(r,G_{p,q})$-covered. The smallest integer $r$ satisfying the above inequality  is \[\widetilde{r}=p+\left\lfloor \left(\sqrt{q+\frac{1}{8}}-\frac{\sqrt{2}}{2}\right)^2-\frac{1}{8} \right\rfloor. \] Since any permutation $f\in S_n$ is $(\widetilde{r},G_{p,q})$-covered, we have $r(G_{p,q})\leq \widetilde{r}$.

\end{proof}

 The following theorem shows that the upper bound in Lemma~\ref{upper} can be achieved.  We  prove it by giving a constructive lower bound.

\begin{theorem}\label{rGpq}
For all $p,q\in \mathbb{N^+}$ and $p\geq q$,
\begin{equation*}
r\left(G_{p,q}\right)= p+\left\lfloor \left(\sqrt{q+\frac{1}{8}}-\frac{\sqrt{2}}{2}\right)^2-\frac{1}{8} \right\rfloor.
\end{equation*}
\end{theorem}
\begin{proof} By simple verification, $r\left(G_{p,q}\right)=p$ when $q=1$ and $2$, agreeing with the claim. Therefore, we assume that $q\geq 3$.

For convenience, let $r_0=p+\left\lfloor \left(\sqrt{q+\frac{1}{8}}-\frac{\sqrt{2}}{2}\right)^2-\frac{1}{8} \right\rfloor-1$.
By Lemma~\ref{upper}, it suffices to show that there exists a permutation $f_0 \in S_{p+q}$, such that $f_0$ is $r_0$-exposed by $G_{p,q}$, i.e., for any permutation $g \in G_{p,q}$, we have $d(f_0 ,g)>r_0$.

Let $k=p+q-r_0>1$. Before  constructing the permutation $f_0\in S_{p+q}$, we define a sequence of numbers $\lambda(i)$, $i\in [k]$, as follows:

\begin{equation*}
\lambda(i)=\begin{dcases}
             q, &i=1,  \\
             k(i-1)-{i \choose 2}, &i\in [2, k].
             \end{dcases}
\end{equation*}
We study some properties of this sequence. First, $\lambda(2)=k-1=q-\left\lfloor q+\frac{1}{2}-\sqrt{2q+\frac{1}{4}}\right\rfloor<q$. Since $\lambda(i+1)-\lambda(i)=k-i$, the sequence $\lambda(i)$ is strictly increasing when $i\in [2, k]$. Now we want to know how large is $\lambda(k)$.  From the proof of Lemma~\ref{upper}, we know that $r_0+1$ is the smallest integer satisfying the inequality
\begin{displaymath}
\frac{(p+q-r)(p+q-r-1)}{2}<q.
\end{displaymath}
So $\lambda(k)=k(k-1)/2=(p+q-r_0)(p+q-r_0-1)/2\geq q$. This means the sequence $\lambda(i)$, $i\in [k]$ ends at a number at least $q$.

Since  $\lambda(2)<q$, we  denote $I$  the largest number $i\in[2,k]$ such that $\lambda(i)<q$. Note that $I<k$. Now we give the permutation $f_0\in S_{p+q}$ by defining values on some selected positions.
 \begin{equation}\label{pqf0}
f_0(j)= \begin{dcases}
            i, &  \textrm{if } j=p+\lambda(i)  \textrm{ for } i\in [I],  \\
             \textrm{arbitrary}, &\textrm{otherwise}.
             \end{dcases}
\end{equation}
It is easy to check that the permutation $f_0$ is well defined.

Let $C=G_{p,q}$. We next check that for any permutation $g\in C$, we have  $d(f_0,g)>r_0$, i.e., $f_0$ is $r_0$-exposed by $C$. By Lemma~\ref{iff}, it is sufficient to prove that $[p+1,p+q]= \cup_{j\in[p+1,p+q]}A_{j\to f_0(j)}^{C}$. Here, it is enough to check the union of $A_{j\to f_0(j)}^{C}$ when $j$ is defined in (\ref{pqf0}).

For each selected position $j=p+\lambda(i)$,  $i\in [I]$, we first find out the possible values $g(j)$ such that $j\to i$ is $r_0$-exposed by $g$. Once the values $g(j)$ are fixed, for some $j\in [p+1,p+q]$,  the locations of the element $p+1$ in these permutations $g\in G_{p,q}$ can be determined easily.  When $i=1$, $f_0(p+q)=1$, so $g(p+q)\in [r_0+1,p+q]$. Hence \[A^C_{p+q\to 1}= [p+q-r_0-1] +p=[\lambda(2)]+p.\] Here, $S+x\triangleq\{a+x: a\in S\}$ for any set $S$ and element $x$.
 When $i\in [2, I]$,  $f_0(j)=i$, we have $g(j)\in [r_0+i+1,p+q]$, which is well defined due to the fact that  $i<k$. If $g(p+\lambda(i))=p+q$, then $g(p+\lambda(i)+1)=p+1$; in general if  $g(p+\lambda(i))=p+q-t$, $t\leq p+q-r_0-i-1$, then $g(p+l)=p+1$, where $l=(\lambda(i)+t+1)\mm q$. Note that when $t=p+q-r_0-i-1$, $l=\lambda(i)+p+q-r_0-i=\lambda(i+1)<q$ when $i<I$. Hence, we have for $i\in [I-1]$,
 \[A^C_{p+\lambda(i)\to i}=[\lambda(i)+1,\lambda(i+1)] +p.\]
 When $i=I$, since $\lambda(i+1)\geq q$, we have $[\lambda(i)+1,q] +p\subset A^C_{p+\lambda(i)\to i}$.

Finally, combining all pieces together, we have
\begin{align*} [p+1,p+q] \supseteq & \bigcup_{j\in[p+1,p+q]}A_{j\to f_0(j)}^{C}\supseteq \bigcup_{i\in[I]}A_{p+\lambda(i)\to i}^C \\\supseteq & [\lambda(2)]\cup[\lambda(2)+1,\lambda(3)]\cup\cdots\cup[\lambda(I)+1,q]+p \\=&[p+1,p+q], \end{align*} which completes the proof.
\end{proof}

\section{Relabeling the (p,q)-Type Group}\label{s:rel}
The labelling problem of permutation group codes was introduced by Tamo and Schwartz \cite{tamo2012labeling} when considering the following scenario. If $C$ and $C'$ are conjugate subgroups of the symmetric group, then from a group-theoretic point of view, they are almost the same algebraic object, which may share the same encoding or even decoding algorithm. However, from a coding point of view, these two codes can possess vastly different minimal distance, which is one of the most important properties of a code. Hence, given a certain group code, a labelling problem is to choose an isomorphic conjugate of the group, having the same group-theoretic structure, but with higher minimal distance. It was shown \cite{tamo2012labeling} that the minimum $l_\infty$-distance of some codes
could move from the
minimum possible $1$, to the maximum possible $n-1$ after relabeling.

In \cite{karni2018infinity}, the authors considered the labelling problem for the covering radius of permutation codes $G_n$. In contrast to the variety of minimal distance, they showed that the covering radii of transitive cyclic groups are quite robust. In particular, the maximum radius after relabeling is $n-\lceil\frac{\sqrt{4n+1}-1}{2}\rceil$, which does not  increase the value of $r(G_n)$ by more than one, while the minimum radius can neither reduce $r(G_n)$ by much.

In this section, we study the covering radius of the $(p,q)$-type groups after relabeling. Let $C=G_{p,q}$, recall that a relabeling of $C$ by conjugation of $\pi\in S _{p+q}$ is defined as $C^{\pi}\triangleq \pi C\pi^{-1} = \{\pi g\pi ^{-1}:g\in C\}$, and the maximum and  minimum radii after relabeling  are denoted by $L_{max}(C)\triangleq \max_{\pi\in S_{p+q}}r(C^{\pi})$ and $L_{min}(C)\triangleq \min_{\pi\in S_{p+q}}r(C^{\pi})$, respectively.


Note that for each $\pi\in S_{p+q}$,  $C^{\pi}=\left<(\pi(1),\pi(2),\ldots,\pi(p)),
 (\pi(p+1),\ldots,\pi(p+q))\right>$. The labeled permutation in $C^{\pi}$ has the same cycle structure as in $C$ but the elements within each cycle are relabeled by $\pi$. Denote \[H\triangleq\{\pi(1),\ldots,\pi(p)\}\text{ and  }R\triangleq\{\pi(p+1),\ldots,\pi(p+q)\}.\] For convenience, we view the set $C^{\pi}$ as a direct product of the following two circulant arrays $\H$ and $\R$ over alphabets $H$ and $R$,  with specified sets of locations $H$ and $R$, respectively, as shown in Fig~\ref{tab1}. Note that for each $i\in [q]$,  there are $p$ permutations $g$ in $C^{\pi}$  that are consistent with the $i$th row of $\R$, namely, $g(l)=\R(i,l)$ for all $l\in R$.  Similarly, for each $j\in [p]$, there are $q$ permutations $g$ in $C^{\pi}$  that are consistent with the $j$th row of $\H$, namely, $g(l)=\H(j,l)$ for all $l\in H$.

\begin{figure}[h]
\[\begin{array}{ccccccccccccc}
\multicolumn{6}{ c }{\text{location indices of $\H$}} && \multicolumn{6}{ c }{\text{location indices of $\R$}} \\
& \pi(1) & \pi(2)& \cdots& \pi(p-1) & \pi(p)&&& \pi(p+1) & \pi(p+2)& \cdots & \pi(p+q-1)& \pi(p+q)\\
\cline{2-6}\cline{9-13}
\multirow{4}{*}{$\H$:}
& \multicolumn{1}{|c  }{\pi(1)} & \pi(2)& \cdots & \pi(p-1) &  \multicolumn{1}{c | }{\pi(p)}& &\multirow{4}{*}{$\R$:}  & \multicolumn{1}{|c  }{\pi(p+1)} & \pi(p+2)& \cdots & \pi(p+q-1)& \multicolumn{1}{c | }{\pi(p+q)}\\
& \multicolumn{1}{|c  }{\pi(2)} & \pi(3)& \cdots & \pi(p) & \multicolumn{1}{c | }{\pi(1)}& && \multicolumn{1}{|c  }{\pi(p+2)} & \pi(p+3)& \cdots & \pi(p+q)& \multicolumn{1}{c | }{\pi(p+1)}\\
& \multicolumn{1}{|c  }{\vdots} & \vdots& \cdots & \vdots& \multicolumn{1}{c | }{\vdots}&& &  \multicolumn{1}{|c  }{\vdots} & \vdots& \cdots & \vdots& \multicolumn{1}{c | }{\vdots} \\
& \multicolumn{1}{|c  }{\pi(p)} & \pi(1)& \cdots & \pi(p-2) & \multicolumn{1}{c | }{\pi(p-1)}& &&  \multicolumn{1}{|c  }{\pi(p+q)} & \pi(1)& \cdots & \pi(p+q-2)& \multicolumn{1}{c | }{\pi(p+q-1)}\\
\cline{2-6}\cline{9-13}
\end{array}\]
    \caption{The arrays $\H$ and $\R$ are shown in the boxes. The numbers above the boxes are the corresponding locations. The notation $\R(i,l)$ means the entry in the $l$th position of the $i$th row in $\R$. For example, $\R(2,\pi(p+2))=\pi(p+3)$.}
     \label{tab1}
\end{figure}

   We will determine the maximum covering radius $L_{max}(C)$ exactly by using similar technique as in Section~\ref{s:gpq}.
 First, we need a definition of $A$ similar to Def~\ref{extended definition of A}.

  \begin{definition} Let $C=G_{p,q}$. For any permutations $f$ and $\pi$ in $ S_{p+q}$, $i\in[p+q]$, let
   \begin{equation*}
A_{i\to f(i)}^{C^\pi}=\begin{dcases}
\{g^{-1}(\pi(1)): i\to f(i) \text{ is $r$-exposed by  $g$ in $C^\pi$} \}, &i\in H;\\
\{g^{-1}(\pi(p+1)): i\to f(i) \textrm{ is $r$-exposed by  $g$ in $C^\pi$}\}, &i\in R.
\end{dcases}
\end{equation*}
  \end{definition}

Simple observations that are similar to Lemmas~\ref{counting} and~\ref{iff}, are given below without proof.  Recall that  $B\triangleq [p+q-r-1]$    and   $T\triangleq[r+2,p+q]$. Here, the dependence of $B$ and $T$ on $p$, $q$ and $r$ is implicit.

\begin{lemma}\label{AA}
Let $p$, $q$ and $r$ be integers such that $p\geq q$ and $\frac{p+q}{2}-1<r<p+q$. Let $C=G_{p,q}$. For any $\pi\in S_{p+q}$, we have
\begin{equation*}
|A_{i\to j}^{C^\pi}|\leq \begin{dcases}
             p+q-r-j, &j\in B,  \\
             j-r-1, &j\in T,\\
             0, &\text{otherwise}.
             \end{dcases}
\end{equation*}
\end{lemma}

\begin{lemma}\label{if and only if A}
Let $f\in S_{p+q}$ be any permutation and $C^\pi\subseteq S_{p+q}$ be a group of $(p,q)$-type, $\pi\in S_{p+q}$. Then  $f$ is $(r,C^\pi)$-exposed if and only if $H= \bigcup_{i\in H} A^{C^\pi}_{i \to f(i)}$ or $R= \bigcup_{i\in R} A^{C^\pi}_{i \to f(i)}$.
\end{lemma}

Now we give an upper bound of $L_{max}(C)$.

\begin{lemma}\label{lmaxup}
Let $C=G_{p,q}$ with $p\geq q$. Then
\begin{equation*}
L_{max}(C)\leq p+\left\lfloor {\left(\sqrt{q+\frac{1}{4}}-\frac{1}{2}\right)}^2\right\rfloor.
\end{equation*}
\end{lemma}
\begin{proof}
 Let $f\in S_{p+q}$ be any permutation. Suppose that the integer $r>\frac{p+q}{2}-1$ and closes to $L_{max}(C)$.  Using Lemma~\ref{AA}  we can get
 \begin{align*} \left |\bigcup_{i\in H}A^{C^\pi}_{i\to f(i)}\right |\leq&\sum _{i\in H} \left |{A^{C^\pi}_{i\mapsto f(i)}}\right |   \notag \\ \leq & \sum _{i\in [p+q]} \left |{A^{C^\pi}_{i\mapsto f(i)}}\right |=  \sum _{f(i)\in B\cup T} \left |{A^{C^\pi}_{i\mapsto f(i)}}\right |\\\leq &2\times\frac{(1+p+q-r-1)(p+q-r-1)}{2}\triangleq d(r). \end{align*}

For the same reason, it's not hard to get

\begin{equation*}
\left|\bigcup_{i\in R}A^{C^\pi}_{i\to f(i)}\right|\le d(r).
\end{equation*}

 By Lemma~\ref{if and only if A}, if  $d(r)<q$, then $f$ is $(r,C^\pi)$-covered, i.e., $L_{max}(C)$ is upper bounded by the smallest integer $r$ satisfying $d(r)<q$. We only need to solve the quadratic inequality $d(r)<q$ for $r$, the result of which is $r>p+q-\frac{1}{2}-\sqrt{q+\frac{1}{4}}$.
 Notice that $p+\left\lfloor {\left(\sqrt{q+\frac{1}{4}}-\frac{1}{2}\right)}^2\right\rfloor$ is the least integer which is larger than $p+q-\frac{1}{2}-\sqrt{q+\frac{1}{4}}$, so $L_{max}(C)\leq p+\left\lfloor {\left(\sqrt{q+\frac{1}{4}}-\frac{1}{2}\right)}^2\right\rfloor$.
 \end{proof}


The next lemma shows that when $q\ge 3$ the upper bound in Lemma~\ref{lmaxup} is tight.

\begin{theorem}\label{lmax}
Let $C=G_{p,q}$ with $p\geq q\geq 3$. Then \begin{equation*}
L_{max}(C)= p+q-\left\lceil \frac{\sqrt{4q+1}-1}{2}\right\rceil.\end{equation*}
\end{theorem}

\begin{proof}
 Note that $p+q-\left\lceil \frac{\sqrt{4q+1}-1}{2}\right\rceil$ shares the same value with $p+\left\lfloor {\left(\sqrt{q+\frac{1}{4}}-\frac{1}{2}\right)}^2\right\rfloor$, which is the upper bound of $L_{max}(C)$ by Lemma~\ref{lmaxup}. So we only need to show that $L_{max}(C)\geq  p+q-\left\lceil \frac{\sqrt{4q+1}-1}{2}\right\rceil$. Let $r_0=p+q-\left\lceil \frac{\sqrt{4q+1}-1}{2}\right\rceil-1$, it suffices to find a permutation $\pi\in S_{p+q}$, and a permutation $f_0\in S_{p+q}$, such that $f_0$ is $r_0$-exposed by $C^\pi$.

 For $q\leq 5$, let $\pi\in S_{p+q}$ be a permutation satisfying $\pi(p+i)=i$, $i\in [2]$ and $\pi(p+i)=p+i$, $i\in [3,q]$. The corresponding permutation $f_0$  that  is $r_0$-exposed by $C^\pi$ has the following constraints: when $q=3$, $f_0(1)=p+3$ and $f_0(2)=p+2$; when $q=4$, $f_0(1)=1, f_0(2)=p+4$ and $f_0(p+3)=2$; when $q=5$, $f_0(1)=1, f_0(2)=p+5$, $f_0(p+3)=2$ and $f_0(p+5)=p+4$.

When $q\geq 6$, we prove it in two cases. Denote $k=\left\lceil \frac{\sqrt{4q+1}-1}{2}\right\rceil$, which is the smallest integer satisfying $k^2+k-q\ge0$. Hence $(k-1)^2+(k-1)=k^2-k<q$.

\textbf{Case 1: }when $q=k(k+1)$.
 We find a permutation $\pi$ with a form like this:
\begin{equation*}
\pi(i)=\begin{dcases}
       2, & \textrm{when $i=p+1$,}\\
             1, & \textrm{when $i=p+2$,}\\
             i-p,&\textrm{when $ i\in[p+3,p+k]$,}\\
             i, &\textrm{when $i\in[p+k+1,p+q]$,}\\
             \textrm{arbitrary}, & \textrm{otherwise}.
             \end{dcases}
\end{equation*}
In this case, $C^\pi=\left<(\pi(1),\pi(2),\ldots,\pi(p)),(2, 1,3,4,\ldots, k, p+k+1,\ldots,p+q)\right>$. Let $R=\{1,2,3,\ldots,k,p+k+1,\ldots,p+q\}$.  Note that for any element $s_1, s_2\in R$, $s_1\not\equiv s_2\pmod p$.

Now we define our permutation $f_0$ as follows:
\begin{equation*}
f_0(i)=
\begin{dcases}
             1, & \textrm{when $i=1$,}\\
             p+q, & \textrm{when $i=2$,}\\
             p+q+1-k,&\textrm{when $i\equiv 3\pmod p$ and $i\in R$,}\\
             k-l,&\textrm{when } i= {l+1 \choose 2}+p+2+k,
              l\in[0,k-2],i\in R ,\\
             p+q-k+1+l,   &\textrm{when } i= p+q-k+2-{l+1 \choose 2},
               l\in[k-2],i\in R,\\
              \textrm{arbitrary}, & \textrm{otherwise}.
 \end{dcases}
\end{equation*}

 When $q=6$, we have $k=2$ and $(\pi(p+1),\ldots,\pi(p+6))=(2,1,p+3,p+4,p+5,p+6)$. In this case, $f_0$ is a permutation such that  $f_0(1)=1$, $f_0(2)=p+6$, $f_0(p+3)=p+5$ and $f_0(p+4)=2$.  To prove  that $f_0$ is $(p+3,C^\pi)$-exposed, we need to compute $A^{C^{\pi}}_{i\to f_0(i)}$ by Lemma~\ref{if and only if A}.  For some $i\in R$, the sets are listed below.
\begin{itemize}
\item $A^{C^{\pi}}_{2\to f_0(2)}=A^{C^\pi}_{2\to p+6}=\{2,p+6\}=\{\pi(p+1),\pi(p+6)\}$;
 \item $A^{C^{\pi}}_{1\to f_0(1)}=A^{C^\pi}_{1\to 1}=\{p+3,p+4\}=\{\pi(p+3),\pi(p+4)\}$;
  \item $A^{C^{\pi}}_{p+3\to f_0(p+3)}=A^{C^\pi}_{p+3\to p+5}=\{1\}=\{\pi(p+ 2)\}$;
  \item $A^{C^{\pi}}_{p+4\to f_0(p+4)}=A^{C^\pi}_{p+4\to 2}=\{p+5\}=\{\pi(p+5)\}$.
\end{itemize}
Combining the above sets, we see that $R=\bigcup_{i\in R}A^{C^\pi}_{i\to f(i)}$, then apply  Lemma~\ref{if and only if A}.

When $q>6$, we have $k\geq 3$, and hence $2k\leq k^2-k<q$. First, we check that $f_0$ is a well defined permutation.
In the definition of $f_0$, we specify the positions of $2k$ distinct elements, i.e. $[k]\cup[p+q-k+1,p+q]$, and the remaining elements are distributed randomly. So we only need to show that the $2k$ positions defined are different. When $l$ varies from $0$ to $k-2$, ${l+1 \choose 2}+p+2+k$ increases strictly with $l$, begins with $p+k+2>p+k$ when $l=0$ and ends at $p+\frac{(k-1)(k-2)}{2}+k+2$. In the range $l\in[k-2]$, the value of $p+q-k+2-{l+1 \choose 2}$ decreases with $l$ strictly, begins with $p+q-k+1<p+q$ and ends at $p+q-k+2-\frac{(k-1)(k-2)}{2}$. Since $q=k^2+k$, we can get $p+q-k+2-\frac{(k-1)(k-2)}{2}$ is strictly larger than $p+\frac{(k-1)(k-2)}{2}+k+2$, which means those $2k$ positions are all distinct, and $f_0$ is well-defined. 

Next, if we can show that $R=\bigcup_{i\in R}A^{C^\pi}_{i\to f(i)}$, then the proof follows by  Lemma~\ref{if and only if A}. So we only need to focus on the locations in $R$, where the elements are also from $R$, see the array $\R$ in Fig~\ref{tab1}. For convenience, let $g_i$ be anyone of the $p$ permutations in $C^{\pi}$ that are consistent with the $i$th row of $\R$, that is,
$g_i(\pi(p+j))=\pi(p+(i+j-1)\mm q)$, $  j\in[q]$.

 When considering the locations $1$, $2$ and $3$, we have the following sets.
\begin{itemize}
\item $A^{C^{\pi}}_{2\to f_0(2)}=A^{C^\pi}_{2\to p+q}=\{2,p+q,p+q-1,\ldots,p+q+2-k\}=
    \{\pi(p+1),\pi(p+q),\ldots,\pi(p+q+2-k)\}$;
 \item $A^{C^{\pi}}_{1\to f_0(1)}=A^{C^\pi}_{1\to 1}=\{3,4,\ldots,k,p+k+1,p+k+2\}=\{\pi(p+3),\pi(p+4)\,\ldots,\pi(p+k+2)\}$;
  \item $\pi(p+2)=1 \in A^{C^{\pi}}_{3\to f_0(3)}$.

\end{itemize}

Besides the three sets above, it remains to show that  $[p+k+3,p+q-k+1]\subset \bigcup_{i\in R}A^{C^\pi}_{i\to f_0(i)}$.

Now we check the locations $ i= {l+1 \choose 2}+p+2+k$, for some $ l\in[0,k-2]$. Here $f_0(i)=k-l$ by definition. When $l=0$, $i=p+k+2$ and $f_0(i)=k$. Since $g_{q-k-1}(i)=g_{q-k-1}(p+k+2)=g_{q-k-1}(\pi(p+k+2))=\pi(p+q)=p+q$, then $f_0$ is $r_0$-exposed by $g_{q-k-1}$ in position $i$, and hence $p+k+3\in A^{C_\pi}_{i\to f_0(i)}$ for $i=p+k+2$. For general $l$, define $\theta(l)=q-k-1-{l+1 \choose 2}\geq 1$. Then $f_0$ is  $r_0$-exposed by $g_{j}$ in position $ i= {l+1 \choose 2}+p+2+k$ for any  $j\in [\theta(l+1)+1, \theta(l)]$, from which we get  $[p+q+2-\theta(l),p+q+1-\theta(l+1)]\subset A^{C^\pi}_{f_0^{-1}(k-l)\to k-l}$. Combining those intervals for $l\in[0,k-2]$ together, we get

\begin{equation*}\left[p+q+2-\theta(0),p+q+1-\theta(k-1)\right]=
\left[p+k+3,p+\frac{k(k-1)}{2}+k+2\right]\subset
\bigcup_{i\in R}A^{C^\pi}_{i\to f_0(i)}.
\end{equation*}

Similarly, we check the locations $i= p+q-k+2-{l+1 \choose 2}$, for $l\in[k-2]$.
               Define $\tau(l)={l+1 \choose 2}+k$.  By computations, we see that $f_0$ is  $r_0$-exposed by $g_{j}$ in position $ i= p+q-k+2-{l+1 \choose 2}$ for any  $j\in [\tau(l), \tau(l+1)-1]$, from which we get  $[p+q+3-\tau(l+1),p+q+2-\tau(l)]\subset A^{C^\pi}_{i\to f_0(i)}$. Combining those pieces together for $l\in[k-2]$, we get
\begin{equation*}
\left[p+q+3-\tau(k-1),p+q+2-\tau(1)\right]=
\left[p+q+3-\frac{k(k-1)}{2}-k,p+q-k+1\right]\subset
\bigcup_{i\in R}A^{C^\pi}_{i\to f_0(i)}.
\end{equation*}

Finally, we  need to show that  $p+\frac{k(k-1)}{2}+k+2\geq \left(p+q+3-\frac{k(k-1)}{2}-k\right)-1$, which is true since $k^2+k=q$.

\textbf{Case 2: }when $q\neq k(k+1)$. In this case, we have $k=\left\lceil\frac{\sqrt{4q+1}-1}{2}\right\rceil
=\left\lfloor\frac{\sqrt{4q+1}+1}{2}\right\rfloor$ and $k(k+1)>q$.
So we only need to prove $L_{max}(C)\geq \left\lfloor\frac{\sqrt{4q+1}+1}{2}\right\rfloor$. A simpler pair of permutations $\pi$ and $f_0$ suffices to give the proof. We list them below.
\begin{equation*}
\pi(i)=\begin{dcases}
i-p,&\textrm{when $i\in[p+1,p+k]$,}\\
             i, &\textrm{when $i\in[p+k+1,p+q]$,}\\
             \textrm{arbitrary}, & \textrm{otherwise}.
             \end{dcases}
\end{equation*}
The definition of $f_0$ needs a parameter $I$, which is defined as the smallest integer such that $p+k^2+k-1-{I+1 \choose 2}< p+q$. Since $p+k^2+k-1\geq p+q$ and $p+k^2+k-1-{k+1 \choose 2}=p+{k+1 \choose 2}-1<p+k^2-k<p+q$, we have $1\leq I\leq k$.
\begin{equation*}
f_0(i)=
\begin{dcases}
p+q-k+l,&\textrm{when } i= \pi\left(p+{l+1 \choose 2}\right),
              l\in[k],i\in R ,\\
             k-l+1,   &\textrm{when } i= \pi\left(p+k^2+k-1-{l+1 \choose 2}\right),
               l\in[I,k],i\in R,\\
              \textrm{arbitrary}, & \textrm{otherwise}.
 \end{dcases}
\end{equation*}

First, we claim that $f_0$ is well-defined. When $l$ varies from $1$ to $k$, $p+{l+1 \choose 2}$ increases strictly from $p+1$ to $p+\frac{k(k+1)}{2}\leq p+q$. When $l\in[I,k]$, $p+k^2+k-1-{l+1 \choose 2}$ decreases strictly from some value smaller than $p+q$ to $p+{k+1 \choose 2}-1$. Since $k\geq 3$, $p+{k\choose 2}<p+{k+1\choose2}-1$, we have $f_0$ is well-defined.

Next, we prove that $R=\bigcup_{i\in R}A^{C^\pi}_{i\to f(i)}$, where $R=[k]\cup[p+k+1,p+q]$. Since it is similar to Case 1, we give a sketch of the proof. For $i=\pi\left(p+{l+1 \choose 2}\right)$ for some $l\in[k]$, $A_{i\to p+q-k+l}^{C^\pi}=\{\pi(p+j)|j\in\left[{l \choose 2}+1,{l+1 \choose 2}\right]\}$. We combine those intervals together to get
\begin{equation*}
\left\{\pi(p+j)|j\in\left[1,{k+1 \choose 2}\right]\right\}\subset
\bigcup_{i\in R}A^{C^\pi}_{i\to f_0(i)}.
\end{equation*}
When $i=\pi\left(p+k^2+k-1-{l+1 \choose 2}\right)$ for some $l\in[I+1,k]$, we have \[A_{i\to k-l+1}^{C^\pi}=\left[p+k^2+k-{l+1 \choose 2}, p+k^2+k-1-{l \choose 2}\right].\] Specially, if $l=I$, we have $[p+k^2+k-{I+1 \choose 2}, p+q]\subset A_{i\to k-I+1}^{C^\pi}$. Combining those pieces together, we finish the proof.
\end{proof}

Next we deal with the case when $q\leq 2$ separately.

\begin{lemma}
Let $C=G_{p,q}$ with $p\geq q$. Then $L_{max}(C)=p$ when $q=1$ or $2$.
\end{lemma}
\begin{proof} When $q=1$, $L_{max}(C)=p$ by the fact that $r\left(G_{p,q}\right)=p$ and the upper bound in Lemma~\ref{lmaxup}.

 When $q=2$, we know that $p=r\left(G_{p,q}\right)\leq L_{max}(C) \leq p+1$ by Lemma~\ref{lmaxup}. To show that $L_{max}(C) \leq p$, we need to prove that for any $\pi\in S_{p+2}$ and $f_0\in S_{p+2}$, $f_0$ is $(p,C^\pi)$-covered. Since $r=p$ here, we have $B=\{1\}$ and $T=\{p+2\}$. By Lemma~\ref{AA}, $|A_{i\to 1}^{C^\pi}|\leq 1$ and $|A_{i\to p+2}^{C^\pi}|\leq  1$. So the fact $|H|\geq|R|=2$ implies that $f_0$ is $(p,C^\pi)$-covered if and only if both $|\bigcup_{i\in R} A^{C^\pi}_{i \to f_0(i)}|\leq 1$ and $|\bigcup_{i\in H} A^{C^\pi}_{i \to f_0(i)}|< |H|$ hold. Now we claim that $|\bigcup_{i\in R} A^{C^\pi}_{i \to f_0(i)}|=2$ never happens. Otherwise, $f_0(i)$ must be in $B\cup T=\{1,p+2\}$ for all $i\in R$, and $R$ must be $\{1,p+2\}$, i.e., $\pi(p+1)=1$ and $\pi(p+2)=p+2$. But simple verification shows $|\bigcup_{i\in R} A^{C^\pi}_{i \to f(i)}|=1$ for all possible $f_0$ and $\pi$ with these constraints. The other fact that $|\bigcup_{i\in H} A^{C^\pi}_{i \to f_0(i)}|< |H|$  follows from  similar arguments.
\end{proof}

We don't try to determine $L_{min}(C)$, since we are not able to find a method to give an upper bound estimation. For the lower bound, the sphere packing bound could be a lower bound of $L_{min}(C)$. However, this depends on the volume of a ball with radius $r$, which is not easy to calculate, see \cite{klove2009generating,klove2011lower,KLOVE2008SPHERES,schwartz2017improved,karni2018infinity}. Even if we know the exact volume, the lower bound depends heavily on the relation of $p$ and $q$. When $p$ and $q$ are quite close, the ball-volume method  never works well and it always ends up with a constant as a lower bound when $p+q$ is large enough. Here, we establish a relation between $L_{min}(G_{p,q})$ and $L_{min}(G_p)$, which yields a lower bound from the estimation of $L_{min}(G_p)$.  The gap between the lower bound below and the value of $L_{max}(G_{p,q})$ remains large. New methods are needed to estimate the tightness of the lower bound and to give a nice upper bound estimation.

\begin{lemma}For any $p\geq q$, $L_{min}(G_{p,q})\geq L_{min}(G_p)$. Hence $L_{min}(G_{p,q})\geq p-\left\lceil\sqrt{2p\ln(p)+2p}\right\rceil$.
\end{lemma}
\begin{proof} Suppose that $ L_{min}(G_p)\triangleq r$. To show that $L_{min}(G_{p,q})\geq r$, it is equivalent to show that for any $\pi\in S_{p+q}$, there exists a permutation $f_0\in S_{p+q}$, such that $f_0$ is $(r-1)$-exposed by any permutation in $G_{p,q}^{\pi}$. Now we show how to find the required $f_0$ for any $\pi$.

We first introduce some useful notations. Given any set $D$ of $p$ distinct elements in $[p+q]$, there is a natural bijection from $D$ to $[p]$ by ranking the elements of $D$ in a natural order.  Denote this bijection by $\phi_D$, and denote $\phi_{D}(h):=[\phi_{D}(h_1),\phi_{D}(h_2),\ldots,\phi_{D}(h_p)]\in S_p$  for any {\it permutation vector} $h=[h_1,h_2,\ldots,h_p]$ over $D$. Here, a {\it permutation vector} over $D$ is a vector of $D$ with length $p$ such that each element occurs exactly once in the vector.   Then $\phi_{D}$ induces a bijection from the set of all permutation vectors over $D$ to $S_p$. It is clear that $|\phi^{-1}_D(i)-\phi^{-1}_D(j)|\geq |i-j|$ for any set $D$ of distinct positive elements.

 For any $\pi\in S_{p+q}$, let $\pi|_{[p]}$ be the vector restricted on the positions $[p]$. Then $\pi|_{[p]}\triangleq [\pi(1),\pi(2),\ldots,\pi(p)]$ is a permutation vector over $H=\{\pi(1),\pi(2),\ldots,\pi(p)\}$, and hence $\bar{\pi}=\phi_H(\pi|_{[p]})$ is a permutation in $S_p$.
Since $ L_{min}(G_p)\triangleq r$, there exists a permutation $\bar{f_0}\in S_p$ such that $\bar{f_0}$ is $(r-1)$-exposed by any permutation in $G_{p}^{\bar{\pi}}$.  Now we define $f_0$ as follows. For each $i\in[p]$, let $f_0(\pi(i))=\phi^{-1}_H\left(\bar{f_0}(\bar{\pi}(i))\right)$, and complete the remaining positions to obtain a permutation $f_0\in S_{p+q}$.

We claim that $f_0$ is the desired permutation. In fact, for each permutation $g\in G^{\pi}_{p,q}$, focusing on the locations in $H$, we define $g|_H$ as a permutation vector over $H$, in which the $\bar{\pi}(i)$th entry is $g(\pi(i))$. By the definition of $G^{\pi}_{p,q}$, we know that $\bar{g}=\phi_H(g|_H)$ is a permutation in $G_p^{\bar{\pi}}$ and $\bar{g}(\bar{\pi}(i))=\phi_H(g(\pi(i)))$ for $i\in [p]$. Then there exists a position $l\in [p]$ such that $|\bar{g}(\bar{\pi}(l))-\bar{f_0}(\bar{\pi}(l))|\geq r$.  So \[|g(\pi(l))-f_0(\pi(l))|=\left|\phi^{-1}_H\left(\bar{g}\left(\bar{\pi}(l)\right)\right)-\phi^{-1}_H(\bar{f_0}(\bar{\pi}(l)))\right|\geq \left|\bar{g}(\bar{\pi}(l))-\bar{f_0}(\bar{\pi}(l))\right|\geq r,\]which completes the proof.

%
%
%
\end{proof}


 \section{The Covering Radius of $D_n$}\label{s:dn}
 In  \cite{karni2018infinity}, the authors gave an estimate of the covering radius of $D_n$ as follows.
 \begin{equation}\label{old}
n-\left\lfloor\frac{\sqrt{4n+1}+1}{2}\right\rfloor\ge r(D_n)
\ge\begin{dcases}
             n-\left\lceil\frac{\sqrt{288n+297}-3}{16}\right\rceil, & n\in[4,9],\\
             n-\left\lceil\frac{\sqrt{288n+737}-1}{16}\right\rceil, & n\in[10,911],\\
             n-\left\lceil\frac{\sqrt{18n-18}}{4}\right\rceil, &n\ge912.
\end{dcases}
\end{equation}
The gap between the upper and lower bounds in (\ref{old}) goes to infinity as $n$ grows.
 The upper bound of $r(D_n)$, which coincides with  the covering radius of $G_n$, is trivial but seems too hard to be improved.

In this section, we establish a better lower bound of $r(D_n)$, where the new gap is upper bounded by $1$ for all $n\ge 10$. Firstly we give a weaker lower bound, which shows that the gap is no larger than $2$. No matter what $n\in \mathbb N$ is, with high probability $r(D_n)=r(G_n)$ or $r(D_n)=r(G_n)-1$, and for very rare values of $n$, $r(D_n)$ may be $r(G_n)-2$. We state this result as follows.

\begin{theorem}\label{rdnlb}
For all $n\ge 10$,
 \begin{equation*}
n-\left\lfloor\frac{\sqrt{4n+1}+1}{2}\right\rfloor\ge r(D_n)\ge
n-\left\lceil\frac{\sqrt{4n+13}+1}{2}\right\rceil.
\end{equation*}
\end{theorem}
\begin{proof}
The strategy of our proof is similar to that of Theorem~\ref{rGpq}, but more complicated. First let $r_0=n-\left\lceil\frac{\sqrt{4n+13}+1}{2}\right\rceil-1$. It suffices to show that there exists a permutation $f_0\in D_n$ which is $r_0$-exposed by $D_{n}$.

Let $k=n-r_0-1=\left\lceil\frac{\sqrt{4n+13}+1}{2}\right\rceil\geq 5$, and denote $d_{t}\triangleq{t\choose 2}$ for $t\in[n]$.  Before  constructing the permutation $f_0\in D_n$, we define a sequence of numbers $\lambda(i)$, for some integers $i\in [n]$, as follows:

\begin{equation*}
\lambda(i)=
\left\{
             \begin{array}{ll}
             d_{k}-d_{k-i+1}+1, & i\in[k-1] ,\\
              d_{k}+d_{i-n+k}-2, & i\in[n-k+2,n].\\

             \end{array}
\right.
\end{equation*}

Note that the sequence $\lambda(i)$ will be served as locations of some elements in $f_0$. So we need to check that whether they have repeated or invalid values.
In the range $[k-1]$, we know that $\lambda(1)=1$ and  $\lambda(i)-\lambda(i-1)=k-i+1\geq 2$, so the sequence $\lambda(i)$ is going up to $\lambda(k-2)=d_k-2$ and  $\lambda(k-1)=d_k\leq n$. In the range $[n-k+2,n]$, we have $\lambda(n-k+2)=d_k-1$, $\lambda(n-k+3)=d_k+1$, and $\lambda(i)-\lambda(i-1)=k+i-n-1\geq 2$ for $i\geq n-k+3$. So the sequence $\lambda(i)$ also increases to $\lambda(n)=2d_k-2\geq n+1$, since $k$ is the least positive integer that satisfies $k^{2}-k-2\geq n+1$.

Since  $\lambda(n-k+2)=d_k-1<n$,  we  denote $I$  the largest number $i\in[n-k+2,n]$ such that $\lambda(i)\leq n$. Note that $n-k+2\leq I<n$. Then we define a new value $\lambda'$ as follows, which will be used to replace the value $\lambda(I+1)$ if there is a confliction.
\begin{itemize}
\item[(1)] If $ \lambda(I+1) \mm n $ is different from values $\lambda(i)$ for all $i\in [k-1]\cup[n-k+2,I]$, then let  $\lambda'=(d_{k}+d_{I+1-n+k}-2) \mm n $.
    \item[(2)]If $ \lambda(I+1) \mm n =\lambda(j)$ for some $j\in [k-1]\cup[n-k+2,I]$, then $j$ must belong to $[k-1]$. This follows from the fact that $\lambda(I)\leq n$, and then $\lambda(I+1)=\lambda(I)+k+I-n\leq k+I$, which is less than or equal to $k\leq d_k-4$ (since $k\geq 5$) after taking $\mm n$ operation. Since the only consecutive values in the sequence $\lambda(i)$ for $i\in [k-1]\cup[n-k+2,I]$ are $\lambda(k-2)$, $\lambda(n-k+2)$, $\lambda(k-1)$, $\lambda(n-k+3)$, whose values are $d_k-2$, $d_k-1$, $d_k$, $d_k+1$, respectively, we can increase the value $\lambda(I+1)$ by one, i.e., let $\lambda'=(d_{k}+d_{I+1-n+k}-1) \mm n $.
\end{itemize}

From the definition of $\lambda'$, we can see that the values $\lambda(i)$ for $ i\in [k-1]\cup[n-k+2,I]$, and $\lambda'$ are pairwise distinct values in $[n]$, thus they form a set of well defined locations for $f_0$.
 Now we give the permutation $f_0\in D_n$ by defining values on these selected positions.
 \begin{equation*}\label{pqf}
f_0(j)= \begin{dcases}
            i, &  \textrm{if } j=\lambda(i)  \textrm{ for } i\in [k-1]\cup[n-k+2,I],  \\
             I+1, &  \textrm{if } j=\lambda',  \\
             \textrm{arbitrary}, &\textrm{otherwise}.
             \end{dcases}
\end{equation*}
It is easy to check that the permutation $f_0$ is well defined.
To check that $f_0$ is $r_0$-exposed by $D_n$, we use the one-line notation of permutations in $D_n$. We write \[D_n=\{A_i: i\in [n]\}\cup \{B_i: i\in [n]\},\] where \[A_i=[(i-1) \mm n,(i-2)\mm n ,\ldots,i \mm n]\] and \[B_i=[(n-i+2)\mm n,(n-i+3)\mm n,\ldots,(n-i+1)\mm n].\]

First, we check $f_0$ is $r_0$-exposed by each permutation $A_i$, $i\in [n]$. We focus on the defined positions $\lambda(i)$ of $f_0$. At position $\lambda(1)=1$, we have $f_0(1)=1$, so $f_0$ is $r_0$-exposed by permutations whose value on position $1$ is at least $ r_0+2=n-k+1$, i.e., $(i-1)\mm n\geq n-k+1$. So we get $A_{1},A_{n-k+2},\ldots, A_{n}$ are $r_0$-exposed by $f_0$ at position $\lambda(1)$. For a fixed position $\lambda(j)$, $2\leq j\leq k-1$, we have $f_0(\lambda(j))=j$ and $A_i(\lambda(j))=(i-\lambda(j)) \mm n$. So we need $(i-\lambda(j)) \mm n\geq r_0+1+j=n-k+j$, that is $i\in [\lambda(j-1)+1,\lambda(j)]$. Here the right margin $\lambda(j)$ comes from the fact that $0\mm n=n$.  Similarly, when $j\in [n-k+2,I-1]$, at position $\lambda(j)$, we get that $A_i$ is $r_0$-exposed by $f_0$ for all $i \in [\lambda(j)+1, \lambda(j+1)]$ by solving the inequality $(i-\lambda(j)) \mm n\leq j-(r_0+1)$. If $\lambda(I)=n$, then we have proved that each $A_i$ is $r_0$-exposed by $f_0$. If $\lambda(I)<n$, solving the same inequality for $j=I$, we obtain that $A_i$ is $r_0$-exposed by $f_0$ at position $\lambda(I)$ for all $i \in [\lambda(I)+1, n]$, hence we get the same conclusion.

Next, we check $f_0$ is $r_0$-exposed by each permutation $B_i$, $i\in [n]$.  We prove it by the same strategy. For  a position $\lambda(j)$ with $j\in[k-2]$, we have $B_i(\lambda(j))=(n-i+1+\lambda(j))\mm n$. Solve the inequality  $(n-i+1+\lambda(j))\mm n\geq j+r_0+1$, we get $i\in[\lambda(j)+1, \lambda(j+1)+1]$. For $j\in [n-k+3,I]$, we solve the inequality $j-(r_0+1)\geq(n-i+1+\lambda(j))\mm n$, then we get $i\in [\lambda(j-1),\lambda(j)]$. For $j=I+1$, at the position $\lambda'$, solving the  inequality $I+1-(r_0+1)\geq(n-i+1+\lambda')\mm n$, we find that $i\in [1,\lambda']\cup[\lambda(I)+1,n]$ satisfies the inequality.  Combining the fact that $\lambda(n-k+2)=d_k-1<\lambda(k-1)$, we have proved that for all $i\in [n]$, $B_i$ is $r_0$-exposed by $f_0$.

\end{proof}

Remark: The gap between the upper and lower bound in Eq.(\ref{old}) could be arbitrarily large as $n$ goes to infinity. The lower bound in Theorem~\ref{rdnlb} significantly reduces this gap to $1$ or $2$ for all $n\geq 10$. In fact, only when $n=m(m-1)-1$ or $n=m(m-1)-2$, for some $m\in\mathbb N$, the gap $\left\lceil\frac{\sqrt{4n+13}+1}{2}\right\rceil-\left\lfloor\frac{\sqrt{4n+1}+1}{2}\right\rfloor =2$. For all other values $n$, the gap is just one. The next lemma further reduces the gap to one for all $n$.
\begin{lemma}\label{refinement}
When $n=m(m-1)-2$, $m(m-1)-1$ or $m(m-1)$ for any integer $m\ge 3$, then  $r(D_n)\geq n-m$.
\end{lemma}
\begin{proof}
When $3\leq m\leq 5$, the exact values of $r(D_n)$ are listed in Table \ref{t1}. When $m>5$, we prove it by contradiction. We want to find a permutation $f_0$ far away from every element in $D_n$, i.e., for any given $\pi\in D_n$, $d(f_0,\pi)>n-m-1$. We prove those three cases separately.

When $n=m(m-1)-2$, we define $f_0$ using a location sequence $\lambda$ as follows:
\begin{displaymath}
\lambda(i)=
\left\{
             \begin{array}{ll}
             n-1, & i=1,\\
             d_{m}-d_{m-i+1}+1, & i\in[2,m-1] ,\\
             d_m-d_{m-1}, & i=m,\\
              d_{m}+d_{i-n+m}-2, & i\in[n-m+2,n],\\

             \end{array}
\right.
\end{displaymath}
where $d_{t}\triangleq{t\choose 2}$ for $t\in[n]$. The number $\lambda(i)$ is served as the location of $i$ and we define $f_0$ as follows:
\begin{equation}\label{f0}
f_0(j)=
\left\{
             \begin{array}{ll}
             i, & \textrm{if $j=\lambda(i)$ for some $i\in[m]\cup[n-m+2,n]$,}\\
             \textrm{arbitrary}, & \textrm{otherwise}.\\

             \end{array}
\right.
\end{equation}

For other cases, we define $f_0$ using the similar way as in Eq~(\ref{f0}), but with different location sequences defined on $[m]\cup[n-m+1,n]$. When $n=m(m-1)-1$, we define $\lambda$ as follows:
\begin{displaymath}
\lambda(i)=
\left\{
             \begin{array}{ll}
             n-1, & i=1,\\
             d_{m}-d_{m-i+1}, & i\in[2,m-2] ,\\
             d_m-d_2+1, & i=m-1,\\
             d_m-d_3+1, & i=m,\\
             d_m+d_3-2, & i=n-m+1,\\
             d_m+d_2-2, & i=n-m+2,\\
             d_{m}+d_{i-n+m}-1, & i\in[n-m+3,n].\\

             \end{array}
\right.
\end{displaymath}

When $n=m(m-1)$, we define $\lambda$  as follows:
\begin{displaymath}
\lambda(i)=
\left\{
             \begin{array}{ll}
             n-1, & i=1,\\
             d_{m}-d_{m-i+1}, & i\in[2,m-2] ,\\
             d_m-d_2+1, & i=m-1,\\
             d_m-d_3+1, & i=m,\\
             d_m+d_3-2, & i=n-m+1,\\
             d_m+d_2-2, & i=n-m+2,\\
             d_{m}+d_{i-n+m}-1, & i\in[n-m+3,n-2],\\
               d_{m}+d_{i-n+m}, & i=n-1,\\
             n, & i=n.\\

             \end{array}
\right.
\end{displaymath}

The method to check that $f_0$ is $(n-m-1)$-exposed by every permutation in $D_n$ is much the same as the method we use in Theorem~\ref{rdnlb}. We leave it to the readers.
\end{proof}

Lemma~\ref{refinement} improves the lower bound of the covering radius $r(D_n)$ by one for all $n$ of the special forms. Combining Lemma~\ref{refinement} and Theorem~\ref{rdnlb},  we get the following result.

\begin{theorem}\label{fin}
For all integer $n\ge 3$, we have
\begin{displaymath}
n-\left\lfloor\frac{\sqrt{4n+1}+1}{2}\right\rfloor\ge r(D_n)\ge n-\left\lfloor\frac{\sqrt{4n+1}+1}{2}\right\rfloor-1.
\end{displaymath}
Specially, if there exists some integer $m>0$ such that $n=m(m-1)$, we know the exact value of $r(D_n)=n-m$.
\end{theorem}

 \subsection{Efficient algorithms of $r(D_n)$}

In Table \ref{t1}, we list the exact values of $r(D_n)$ for $n\le 20$ which are determined by computer search. Here, the subscript $u$ means the exact value achieves the upper bound of Theorem~\ref{fin}, the subscript $l$ means achieving the lower bound, and $e$ means the exact value achieves both the upper and the lower bounds.

\begin{table}[h]
\centering
\begin{tabular}{cccccccccc}
\hline
$n$&3&4&5&6&7&8&9&10&11\\
\hline
$r(D_n)$&$0_l$&$1_l$&$2_l$&$3_e$&$4_u$&$5_u$&$5_l$&$6_l$&$7_l$\\
\hline
$n$&12&13&14&15&16&17&18&19&20\\
\hline
$r(D_n)$&$8_e$&$9_u$&$10_u$&$11_u$&$12_u$&$12_l$&$13_l$&$14_l$&$15_e$\\
\hline
\end{tabular}
\caption{Exact values of $r(D_n)$ for small $n$}
\label{t1}
\end{table}

We now describe our algorithm on determining $r(D_n)$. If we use exhaustive search, we need to compute $O(n!)$ values of $d(f, D_n)$ for each $f\in S_n$, and then output the largest one among them as $r(D_n)$. When $n$ becomes bigger, it takes a very long time that we can not afford to finish the program. In our algorithm, we make use of the two subsets $B=[n-r-1]$ and $T=[r+2,n]$, where $r$ is the lower bound given in Theorem~\ref{rdnlb}. As mentioned in Section~\ref{s:gpq}, only the numbers in $B$ or $T$ can create a difference bigger than $r$ from other numbers in $[n]$. Our algorithm depends on the following observation:

\vspace{0.3cm}
\emph{for any two  permutations $f$ and $f'$ in $S_n$, if $f^{-1}(i)=f'^{-1}(i)$ for all $i\in B\cup T$, then either $d(f, D_n)=d(f', D_n)>r$, or $d(f, D_n)\leq r$ and $d(f', D_n)\leq r$.}
\vspace{0.3cm}

\noindent From the above observation, we only need to take care of the permutations with distinct locations for members in $B\cup T$. Hence, we only need to compute $d(f, D_n)$ for $n(n-1)(n-2)\cdots(n-(2n-2r-2)+1)$ permutations $f$. This greatly reduces the computation time since the lower bound $r$ is very close to $n$.


Note that the above algorithm works for any lower bound $r\leq r(D_n)$. When the lower bound $r$ is not good, we would like to use a bigger number $\tilde{r}>r$ to replace $r$ in our algorithm to reduce the computation time. However,  we don't know this $\tilde{r}$ is a lower bound or not at this time. If it is not, then our algorithm fails to give us the correct answer. We claim that:

\vspace{0.3cm}
\emph{if our algorithm returns a value $\tilde{r}(D_n)$ which is no less than $\tilde{r}$, then $\tilde{r}$ is indeed a lower bound, and hence $\tilde{r}(D_n)$ is the correct covering radius $r(D_n)$.}
\vspace{0.3cm}

\noindent In fact, when we input $\tilde{r}>r$, the subsets $B$ and $T$ become smaller. By our algorithm, this means we compute a smaller set of values $d(f, D_n)$, and among which the maximum value $\tilde{r}(D_n)$ can not exceed the real covering radius $r(D_n)$. So if $\tilde{r}(D_n)\geq \tilde{r}$, which means $\tilde{r}$ is indeed a lower bound of $r(D_n)$, and our algorithm gives us the correct answer.


\section{Conclusion}\label{s:con}
%

In this paper we studied the covering radius of  permutation groups  with $l_\infty$-metric. We determine the covering radius of a $(p,q)$-type group, $G_{p,q}\triangleq \left<(1,2,\cdots,p),(p+1,p+2,\cdots,p+q)\right>$, and the maximum value among the covering radii of all its relabelings. The method we described extends the one used in \cite{karni2018infinity}, and can be used for large groups.

 Given a finite integer $k\ge 1$, let ${p_i}, i\in[k]$ be positive integers with non-increasing order. The  {\em natural $(p_1,p_2,\ldots ,p_k)$-type} group  is defined by $G_{p_1, p_2,\cdots, p_k}\triangleq G_{p_1}\otimes G_{p_2}\otimes\cdots G_{p_k}$,
where $G_{p_i}=\langle(\sum_{s=1}^{i-1}p_s+1, \sum_{s=1}^{i-1}p_s+2,\ldots,\sum_{s=1}^{i}p_s)\rangle$, $i\in[k]$. By the same technique, we obtain

\begin{equation}
r\left(G_{p_1, p_2,\cdots, p_k}\right)= n-p_k+\left\lfloor \left(\sqrt{p_k+\frac{1}{8}}-\frac{\sqrt{2}}{2}\right)^2-\frac{1}{8}\right \rfloor
\end{equation}
and \begin{equation}
L_{max}\left(G_{p_1, p_2,\cdots, p_k}\right)= \begin{dcases} n-\left\lceil \frac{\sqrt{4p_k+1}-1}{2}\right\rceil,& p_k\ge 3,\\
n-p_k, &p_k<3.
\end{dcases}
\end{equation}Details about the proofs of the above results can be provided upon requests.

Another main contribution of this article is that we gave a better lower bound of the covering radius of dihedral group $D_n$, which differs from the upper bound by at most one. This improves the result in \cite{karni2018infinity}, where the gap grows with $n$. Our new result depends on the construction of a permutation that is far from all elements of $D_n$. The algorithm we used to determine $r(D_n)$ for small values of $n$ is very efficient, and works for any permutation group. The experimental results show that both the upper bound and the lower bound maybe tight for $r(D_n)$. We leave this  problem for future study.

\section*{Acknowledgments}
This research is supported by NSFC  under grant 11771419  and by ``the  Fundamental
Research Funds for the Central Universities''.

%


\end{document}